\documentclass[reqno]{amsart}%
\usepackage{amssymb}
\usepackage{amsfonts}
\usepackage{amsmath}
\usepackage{graphicx}
\usepackage{color}%
\providecommand{\U}[1]{\protect\rule{.1in}{.1in}}
\theoremstyle{plain}

\newtheorem{claim}{Claim}

\newtheorem{conjecture}{Conjecture}

\newtheorem{lemma}{Lemma}

\newtheorem{proposition}{Proposition}
\newtheorem{remark}{Remark}

\newtheorem{theorem}{Theorem}
\numberwithin{equation}{section}
\begin{document}
\title[Multiplication on uniform $\lambda$-Cantor sets]{Multiplication on uniform $\lambda$-Cantor sets}
\author{jiangwen Gu}
\address{Department of Mathematics, Ningbo University, 315211, Ningbo, P. R. China}
\email{1811071001@nbu.edu.cn}
\author{kan Jiang}
\address{Department of Mathematics, Ningbo University, 315211, Ningbo, P. R. China}
\email{kanjiangbunnik@yahoo.com, jiangkan@nbu.edu.cn}
\author{lifeng Xi}
\address{Department of Mathematics, Ningbo University, 315211, Ningbo, P. R. China; Key
Laboratory of Computing and Stochastic Mathematics (Ministry of Education),
School of Mathematics and Statistics, Hunan Normal University, Changsha, Hunan
410081, P. R. China}
\email{xilifengningbo@yahoo.com, xilifeng@nbu.edu.cn}
\author{bing Zhao}
\address{Department of Mathematics, Ningbo University, 315211, Ningbo, P. R. China}
\email{zhaobing@nbu.edu.cn}
\thanks{Kan Jiang is the corresponding author. Supported by National Natural Science
Foundation of China (Nos. 11701302, 11831007, 11771226, 11371329, 11301346, 11671147)
and K.C. Wong Magna Fund in Ningbo University.}

\begin{abstract}
Let $C$ be the middle-third Cantor set. Define $C*C=\{x*y:x,y\in C\}$, where
$*=+,-,\cdot,\div$ (when $*=\div$, we assume $y\neq0$).  Steinhaus \cite{HS}
proved in 1917 that
\[
C-C=[-1,1], C+C=[0,2].
\]
In 2019, Athreya, Reznick and Tyson \cite{Tyson} proved that
\[
C\div C=\bigcup_{n=-\infty}^{\infty}\left[ 3^{-n}\dfrac{2}{3},3^{-n}\dfrac
{3}{2}\right] .
\]
In this paper, we give a description of the topological structure and Lebesgue
measure of $C\cdot C$. We indeed obtain corresponding results on the uniform
$\lambda$-Cantor sets.

\end{abstract}
\maketitle

\section{Introduction}

The middle-third Cantor set, denoted by $C$, is a celebrated set in fractal
geometry. Many aspects of this set were analyzed \cite{FG,FW2000,Hutchinson}.
In this paper, we shall investigate the arithmetic on the middle-third Cantor
set. Arithmetic on the fractal sets has some connections to the geometric
measure theory, dynamical systems, and number theory, see
\cite{Yoccoz,BABA1,KarmaCor1,Bond,KCor,KM,DajaniDeVrie,KarmaOomen,Hochman2012,JM,MO,Kan2019,Palis,PS}
and references therein. However, relatively few papers investigate the
structure of the fractal sets under some arithmetic operations. The theme of
this paper is to explore this problem.

Let $C*C=\{x*y:x,y\in C\}$, where $*=+,-,\cdot,\div$ (when $*=\div$, we assume
$y\neq0$). Arithmetic on the middle-third Cantor set was pioneered by
Steinhaus \cite{HS} who proved that
\[
C+C=[0,2],C-C=[-1,1].
\]
It is natural to ask the multiplication and division on $C$. For this
question, recently, Athreya, Reznick and Tyson \cite{Tyson} proved that
$17/21\leq\mathcal{L}(C\cdot C)\leq8/9$, where $\mathcal{L}$ denotes the
Lebesgue measure. Moreover, they also proved that
\[
C\div C=\bigcup_{n=-\infty}^{\infty}\left[ 3^{-n}\dfrac{2}{3},3^{-n}\dfrac
{3}{2}\right] .
\]
Let $f$ be a continuous function defined on an open set $U\subset
\mathbb{R}^{2}$. Denote the continuous image by
\[
f_{U}(C,C)=\{f(x,y):(x,y)\in(C\times C)\cap U\}.
\]
Jiang and Xi \cite{JX2018} proved that if $\partial_{x}f$, $\partial_{y}f$ are
continuous on $U,$ and there is a point $(x_{0},y_{0})\in(C\times C)\cap U$
such that one of the following conditions is satisfied,
\[
1<\left\vert \frac{\partial_{y}f|_{(x_{0},y_{0})}}{\partial_{x}f|_{(x_{0}%
,y_{0})}}\right\vert <3 \mbox{ or } 1<\left\vert \frac{\partial_{x}%
f|_{(x_{0},y_{0})}}{\partial_{y}f|_{(x_{0},y_{0})}}\right\vert <3,
\]
then $f_{U}(C,C)$ has a non-empty interior. In particular, $C\cdot C$ contains
infinitely many intervals. To date, to the best of our knowledge, there are
few results for the structure of $C\cdot C$. This is one of the main
motivations of this paper. We shall give an answer to this question. Let
$K_{\lambda}$ be the uniform $\lambda$-Cantor set generated by the IFS
\[
\{ f_{i}(x)=\lambda x+i\delta\}_{i=0}^{m-1},
\]
where $m\geq2, 0<\lambda<1/m,g=\dfrac{1-m\lambda}{m-1},$ and $\delta
=\lambda+g$. Suppose $\sigma=i_{1}i_{2}\cdots i_{n}\in\{0,1,\cdots m-1\}^{n} $
with its length $|\sigma|=n\geq0$. If $n=0$, then $\sigma=\emptyset$. In this
case, $f_{\sigma}([0,1])=[0,1].$ For any $I=f_{\sigma}([0,1])$, let
$I^{(i)}=f_{\sigma i}([0,1]), 0\leq i\leq m-1$. Denote
\[
I^{(i)}\cdot I^{(j)}=\{xy:x\in I^{(i)}, y\in I^{(j)}\}.
\]
Define
\[
\widehat{I}=\left\lbrace
\begin{array}
[c]{cc}%
\cup_{0\leq p<q}\left( I^{(p)}\cdot I^{(q)}\right)  & \mbox{ if } |\sigma
|\geq1\\
\cup_{1\leq p<q}\left( I^{(p)}\cdot I^{(q)}\right)  & \mbox{ if }
\sigma=\emptyset.
\end{array}
\right.
\]
We denote
\[
\mathcal{A}=\{f_{\sigma}([0,1]) : \sigma=i_{1}i_{2}\cdots i_{n}, i_{1}%
\geq1\}\cup\{[0,1]\}.
\]
Let
\[
K_{R}=\cup_{i=1}^{m-1}f_{i}(K_{\lambda}).
\]
Now we state the first result of this paper.

\begin{theorem}
\label{main1} Let $K_{\lambda}$ be the uniform $\lambda$-Cantor set. If
$1/(m+1)\leq\lambda<1/m$, then
\[
K_{\lambda}\cdot K_{\lambda}= \left( \bigcup_{n=0}^{\infty}\lambda^{n}
(K_{R}\cdot K_{R})\right) \bigcup\left\{ 0\right\} ,
\]
where
\[
K_{R}\cdot K_{R}=\left( \bigcup_{I\in\mathcal{A}}\widehat{I}\right)
\bigcup\{x^{2}:x\in K_{R}\}.
\]

\end{theorem}

\begin{remark}
When $\lambda=1/3, m=2$, we give the structure of $C\cdot C$. With the help of
computer program, we are able to calculate the value of the Lebesgue measure
of $C\cdot C$, which is about $0.80955.$
\end{remark}

\begin{remark}
The main ingredient of $K_{R}\cdot K_{R}$ is $\bigcup_{I\in\mathcal{A}%
}\widehat{I}$ as the Lebesgue measure of $\{x^{2}:x\in K_{R}\}$ is zero. We
shall use this fact in the proof of next result. It is natural to consider
that whether some points in $\{x^{2}:x\in K_{R}\}$ can be absorbed by
$\bigcup_{I\in\mathcal{A}}\widehat{I}$. We may prove, for instance, that for
$m=2$, if $0.44<\lambda<1/2$, then
\[
K_{R}\cdot K_{R}\setminus\bigcup_{I\in\mathcal{A}}\widehat{I}=\{(1-\lambda
)^{2},1\}.
\]
Equivalently, in this case
\[
\bigcup_{I\in\mathcal{A}}\widehat{I}=((1-\lambda)^{2},1).
\]

\end{remark}

The next result is to consider the function $\Phi(\lambda)=\mathcal{L}%
(K_{\lambda}\cdot K_{\lambda})$. In terms of some basic results in analysis,
we prove the following result.

\begin{theorem}
\label{main2} If $\lambda\in[1/(m+1),1/m),$ then the function $\mathcal{L}%
(K_{\lambda}\cdot K_{\lambda})$ is continuous.
\end{theorem}

This paper is arranged as follows. In section 2, we give the proofs of main
theorems. In section 3, we give some remarks and problems.

\section{Proof of Main results}
\noindent In this section, we shall give the proofs of the main theorems.
\subsection{Proof of Theorem \ref{main1}}
Let $E=[0,1]$. For any $(i_{1}\cdots i_{n})\in\{0,1,\cdots m-1\}^{n}$, we call
$f_{i_{1}\cdots i_{n}}(E)$ a basic interval with length $\lambda^{n}$. Denote
by $E_{n}$ ($n\geq1$) the collection of all the basic intervals with length $\lambda^{n}%
$. Let $J=f_{\sigma}([0,1])\in E_{n}$ be a basic interval for some $\sigma\in\{0,1,\cdots, m-1\}^n$. Denote $\widetilde{J}=\cup_{i=0}^{m-1}f_{\sigma i}([0,1])$.  Let
$[A,B]\subset[0,1]$, where $A$ and $B$ are the left and right endpoints of
some basic intervals in $E_{k}$ for some $k\geq1$, respectively. $A$ and $B$
may not be in the same basic interval.  Let $F_{k}$ be the collection of all the basic intervals
in $[A,B]$ with length $\lambda^{k}, k\geq k_{0}$ for some $k_{0}\in
\mathbb{N}^{+}$, i.e. the union of all the elements of $F_{k}$ is denoted by
$G_{k}=\cup_{i=1}^{t_{k}}I_{k,i}$, where $t_{k}\in\mathbb{N}^{+}$, $I_{k,i}\in
E_{k}$ and $I_{k,i}\subset[A,B]$. Clearly, by the definition of $G_{n}$, it
follows that $G_{n+1}\subset G_{n}$ for any $n\geq k_{0}.$
Similarly, let $M$ and $N$ be the left and right endpoints of
some basic intervals in $E_{k}$ for some $k\geq k_0$. Denote by $G_{k}^{\prime}$  the union of all the basic intervals with length $\lambda^k$ such that all of these basic intervals are subsets of $[M,N]$, i.e.
$G_{k}^{\prime}=\cup_{i=1}^{t_{k}^{\prime}}I_{k,i}^{\prime}$, where $t_{k}^{\prime}\in\mathbb{N}^{+}$, $I_{k,i}^{\prime}\in
E_{k}$ and $I_{k,i}^{\prime}\subset[M,N]$.
Now, we state a key lemma of our paper.
\begin{lemma}
\label{key1} Assume $F:\mathbb{R}^{2}\to\mathbb{R}$ is a continuous function.
Suppose $A$ and $B$ ($M$ and $N$) are the left and right endpoints of some basic intervals
in $E_{k_{0}}$ for some $k_{0}\geq1$, respectively. Then $K_{\lambda}%
\cap[A,B]=\cap_{n={k_{0}}}^{\infty}G_{n}$ ($K_{\lambda}%
\cap[M,N]=\cap_{n={k_{0}}}^{\infty}G_{n}^{\prime}$). Moreover, if for any $n\geq k_{0}$
and any basic intervals $I_{1}\subset G_{n}$, $ I_{2}\subset G_{n}^{\prime}$
\[
F(I_{1}, I_{2})=F(\widetilde{I_{1}}, \widetilde{I_{2}}),
\]
then $F(K_{\lambda}\cap[A,B],K_{\lambda}\cap[M,N] )=F(G_{k_{0}}, G_{k_{0}}^{\prime}).$
\end{lemma}

\begin{proof}
Let $G_{n}=\cup_{i=1}^{t_{n}}I_{n,i}$ for some $t_{n}\in\mathbb{N}^{+}$, where
$I_{n,i}\in E_{n}$ and $I_{n,i}\subset[A,B]$. Then by the construction of
$G_{n}$, i.e. $G_{n+1}\subset G_{n}$ for any $n\geq k_{0}$, it follows that
\[
K_{\lambda}\cap[A,B]=\cap_{n=k_{0}}^{\infty}G_{n}.
\]
Analogously, we can prove that
\[
K_{\lambda}\cap[M,N]=\cap_{n=k_{0}}^{\infty}G_{n}^{\prime}.
\]
By the continuity of $F$, we conclude that
\begin{align}
\label{identity}F(K_{\lambda}\cap[A,B],K_{\lambda}\cap[M,N])=\cap_{n=k_{0}%
}^{\infty}F(G_{n}, G_{n}^{\prime}).
\end{align}
By virtue of the relation $G_{n+1}=\widetilde{G_{n}}$  ($G_{n+1}^{\prime}=\widetilde{G_{n}^{\prime}}$) and the condition in the
lemma, we have
\begin{align*}
F(G_{n}, G_{n}^{\prime}) & = \cup_{1\leq i\leq t_{n},1\leq j\leq t_{n}^{\prime}}F(I_{n,i}, I_{n,j}^{\prime})\\
& = \cup_{1\leq i\leq t_{n},1\leq j\leq t_{n}^{\prime}}F(\widetilde{I_{n,i}},\widetilde{I_{n,j}^{\prime}})\\
& = F(\cup_{1\leq i\leq t_{n}}\widetilde{I_{n,i}},\cup_{1\leq j\leq t_{n}^{\prime}%
}\widetilde{I_{n,j}^{\prime}})\\
& =F(G_{n+1}, G_{n+1}^{\prime}).
\end{align*}
Therefore, $F(K_{\lambda}\cap[A,B],K_{\lambda}\cap[M,N] )=F(G_{k_{0}},
G_{k_{0}}^{\prime})$ follows immediately from identity (\ref{identity}) and $F(G_{n},
G_{n}^{\prime})=F(G_{n+1}, G_{n+1}^{\prime})$ for any $n\geq k_{0}.$
\end{proof}

\noindent Recall that $\delta=\lambda+g$ and $g=\frac{1-m\lambda}{m-1}$ with $m\geq2.$

\begin{proposition}
\label{pro1} Let $\lambda\in\lbrack\frac{1}{m+1},\frac{1}{m})$. Suppose
$I\in\mathcal{A}$ and $I^{(0)},\cdots,I^{(m-1)}$ are basic subintervals of $I$
from left to right such that $|I^{(i)}|/|I|=\lambda$ for all $0\leq i\leq
m-1.$ Then for any $p<q,$ we have
\[
(I^{(p)}\cap K_{\lambda})\cdot(I^{(q)}\cap K_{\lambda})=I^{(p)}\cdot I^{(q)}.
\]

\end{proposition}

\begin{proof}
Given basic subintervals $I^{\ast}\subset I^{(p)}$ and $J^{\ast}\subset
I^{(q)}$ with $I^{\ast}=[a,a+t]$ and $J^{\ast}=[b,b+t]$ ($b\geq a$), we
suppose that $A^{(0)},\cdots,A^{(m-1)}$ $(B^{(0)},\cdots,B^{(m-1)})$ are basic
subintervals of $I^{\ast}$ ($J^{\ast}$ respectively) from left to right.

For rectangles $R_{1}=[a_{1},b_{1}]\times\lbrack a_{2},b_{2}]$ and
$R_{2}=[c_{1},d_{1}]\times\lbrack c_{2},d_{2}]$ with $a_{1},a_{2},c_{1}%
,c_{2}\geq0,$ we denote by $R_{1}\rightarrow R_{2}$ if $b_{1}b_{2}\geq
c_{1}c_{2}.$ By Lemma \ref{key1}, it suffices to verify that
\begin{align*}
& \ \ \ \ \ (A^{(0)}\times B^{(0)})\\
& \rightarrow(A^{(0)}\times B^{(1)})\rightarrow(A^{(1)}\times B^{(0)})\\
& \rightarrow(A^{(0)}\times B^{(2)})\rightarrow(A^{(1)}\times B^{(1)}%
)\rightarrow(A^{(2)}\times B^{(0)})\\
& \rightarrow(A^{(0)}\times B^{(3)})\rightarrow(A^{(1)}\times B^{(2)}%
)\rightarrow(A^{(2)}\times B^{(1)})\rightarrow(A^{(3)}\times B^{(0)})\\
& \cdots\\
& \rightarrow(A^{(0)}\times B^{(m-1)})\rightarrow(A^{(1)}\times B^{(m-2)}%
)\rightarrow\cdots\rightarrow(A^{(m-2)}\times B^{(1)})\rightarrow
(A^{(m-1)}\times B^{(0)})\\
& \rightarrow(A^{(1)}\times B^{(m-1)})\rightarrow(A^{(2)}\times B^{(m-2)}%
)\rightarrow\cdots\rightarrow(A^{(m-1)}\times B^{(1)})\\
& \cdots\\
& \rightarrow(A^{(m-1)}\times B^{(m-1)}).
\end{align*}
For example, please see Fig. 1 when $m=5.$
\begin{figure}[ptbh]\vspace{-0.3cm}
\centering\includegraphics[width=0.34\textwidth]{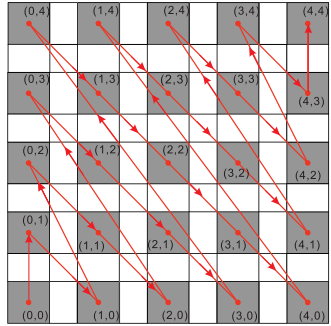}\vspace{-0.3cm}
\caption{$m=5$ and $(i,j)$ represents rectangle $A^{(i)}\times B^{(j)}$}
\end{figure}

\noindent In fact, we need to verify
\begin{equation}
(A^{(i)}\times B^{(j)})\rightarrow(A^{(i+1)}\times B^{(j-1)}),\label{xi1}%
\end{equation}
and
\begin{equation}
(A^{(k)}\times B^{(0)})\rightarrow(A^{(0)}\times B^{(k+1)}%
)\mbox{ and }(A^{(m-1)}\times B^{(i)})\rightarrow(A^{(i+1)}\times
B^{(m-1)}).\label{xi2}%
\end{equation}

\noindent\textbf{(1)} For inequality (\ref{xi1}), let $(a,b)+(i,j)\delta t=(x,y),$ then
$(a,b)+(i+1,j-1)\delta t=(x+\delta t,y-\delta t).$ We only need to show%
\begin{equation}
g(x,y)=(x+\lambda t)(y+\lambda t)-(x+\delta t)(y-\delta t)\geq0.\label{xi3}%
\end{equation}
Note that $x\geq a$, $y\leq b+(1-\lambda)t$ and%
\[
\frac{\partial g(x,y)}{\partial x}=\lambda t+\delta t>0\mbox{ and }\frac
{\partial g(x,y)}{\partial y}=-gt<0.
\]
We obtain that
\begin{align*}
g(x,y)  & \geq g(a,b+(1-\lambda)t)\\
& =t\left(  a\lambda+a\delta+t\lambda^{2}+t\delta^{2}-bg-gt+gt\lambda\right)
=t\cdot u(\lambda,t).
\end{align*}
Since $a\geq\delta$ and $b\leq1$ we have%
\begin{align*}
u(\lambda,t)  & \geq\delta\lambda+\delta^{2}+t\lambda^{2}+t\delta
^{2}-g-gt+gt\lambda\\
& =t\left(  \lambda^{2}+g\lambda+\delta^{2}-g\right)  +(\delta^{2}%
+\lambda\delta-g).
\end{align*}
By $\delta=\lambda+g$ and $g=\frac{1-m\lambda}{m-1},$ one can get%
\begin{align*}
& \lambda^{2}+g\lambda+\delta^{2}-g\\
& =\frac{m^{2}\lambda-m\lambda^{2}-m+2\lambda^{2}-3\lambda+2}{\left(
m-1\right)  ^{2}}=\frac{h(\lambda)}{\left(  m-1\right)  ^{2}}.
\end{align*}
Note that $\lambda<\frac{1}{m}\leq1/2,$ then%
\[
h^{\prime}(\lambda)=(m^{2}-3)-2\lambda\left(  m-2\right)  \geq m(m-1)-1\geq0.
\]
Hence for any $\lambda\geq\frac{1}{m+1}$ and any $m\geq2,$
\[
\frac{h(\lambda)}{\left(  m-1\right)  ^{2}}\geq\frac{h(\frac{1}{m+1})}{\left(
m-1\right)  ^{2}}=\frac{1}{\left(  m^{2}-1\right)  ^{2}}\left(  m^{2}%
-m+1\right)  >0.
\]
On the other hand, we obtain%
\begin{align*}
\delta^{2}+\lambda\delta-g  & =\frac{1}{\left(  m-1\right)  ^{2}}\left(
\left(  2-m\right)  \lambda^{2}+\left(  m^{2}-3\right)  \lambda+\left(
2-m\right)  \right) \\
& =\frac{1}{\left(  m-1\right)  ^{2}}v(\lambda).
\end{align*}
If $m=2$, then $\delta^{2}+\lambda\delta-g \geq 0$. If $m=3$, then
we  note that the symmetric axis of the graph $w=v(\lambda)$ is
\[
\lambda=\frac{m^{2}-3}{2(m-2)}\geq\frac{1}{2}\frac{m^{2}-4}{m-2}=\frac{m+2}%
{2}\geq2.
\]
which implies%
\begin{align*}
\inf_{\lambda\in\lbrack\frac{1}{m+1},\frac{1}{m})}v(\lambda)  & \geq
\min\left(  v(\frac{1}{m}),v(\frac{1}{m+1})\right) \\
& =\min\left(  \frac{2\left(  m-1\right)  ^{2}}{m^{2}},\frac{m^{2}%
-m+1}{\left(  m+1\right)  ^{2}}\right)  >0.
\end{align*}
Then (\ref{xi3}) follows.

\noindent\textbf{(2)} For inequality (\ref{xi2}), we will need the following

\begin{claim}
\label{claim00}If $x_{1}+y_{1}\geq x_{2}+y_{2}$ with $x_{1}<y_{1},$
$x_{2}<y_{2}$ and $x_{1}>x_{2},$ then we have $x_{1}y_{1}\geq x_{2}y_{2}.$
\end{claim}

\noindent In fact, let $H(x)=x(c-x)$ with $c=x_{1}+y_{1},$ then we have%
\[
x_{2}y_{2}\leq x_{2}(y_{2}+(x_{1}+y_{1})-(x_{2}+y_{2}))=H(x_{2})<H(x_{1})
\]
since $x_{2}<x_{1}<c/2.$

Let $(x_{1},y_{1})$ be the upper right corner of the rectangle $A^{(k)}\times
B^{(0)},$ and $(x_{2},y_{2})$ the lower left corner of $A^{(0)}\times
B^{(k+1)},$ then
\[
(x_{1}+y_{1})-(x_{2}+y_{2})=2\lambda t-\delta t=t\frac{(2m-1)\lambda-1}%
{m-1}\geq0
\]
due to $\lambda\geq\frac{1}{m+1}\geq\frac{1}{2m-1}.$ Using Claim
\ref{claim00}, we have $x_{1}y_{1}\geq x_{2}y_{2},$ i.e.,
\[
(A^{(k)}\times B^{(0)})\rightarrow(A^{(0)}\times B^{(k+1)}).
\]
In the same way, we have $(A^{(m-1)}\times B^{(i)})\rightarrow(A^{(i+1)}\times
B^{(m-1)}).$
\end{proof}
\noindent The following lemma is from the fact
\[
K_{\lambda}= \{0\}\cup\left( \bigcup_{n=0}^{\infty}\lambda^{n} ( K_{R})\right)
.
\]

\begin{lemma}
\label{translate}
\[
K_{\lambda}\cdot K_{\lambda}= \{0\}\cup\left( \bigcup_{n=0}^{\infty}%
\lambda^{n} (K_{R}\cdot K_{R})\right) .
\]

\end{lemma}

\begin{proof}
[\textbf{Proof of Theorem \ref{main1}}]$\ $
By Lemma \ref{translate}, it remains to prove
\[
K_{R}\cdot K_{R}=\left( \cup_{I\in\mathcal{A}}\widehat{I}\right) \cup
\{x^{2}:x\in K_{R}\}.
\]
By Proposition \ref{pro1}, it follows that
\[
K_{R}\cdot K_{R}\supset\left( \cup_{I\in\mathcal{A}}\widehat{I}\right)
\cup\{x^{2}:x\in K_{R}\}.
\]
Conversely, for any $x, y\in K_{R}$, if $x=y$, then $xy\in\{x^{2}:x\in
K_{R}\}$. If $x\neq y$, then there exists some basic interval, denoted by $I$,
such that $x\in I^{(p)}, y\in I^{(q)}$ for $p\neq q$. Therefore, by
Proposition \ref{pro1}, we have that $x\cdot y\in\cup_{I\in\mathcal{A}%
}\widehat{I}.$
\end{proof}

\subsection{Proof of Theorem \ref{main2}}

\begin{lemma}
\label{c1} Let $\lambda\in\left[  1/(m+1),1/m\right)  $. For any $n\geq1$, we
have
\[
\mathcal{L}\left(  (K_{R}\cdot K_{R})\backslash\bigcup\nolimits_{I\in
\mathcal{A},\text{\emph{rank}}(I)\leq k}\widehat{I}\right)  \leq3
\dfrac{(m\lambda)^{k+1}}{1-m\lambda}.
\]

\end{lemma}

\begin{proof}
First, note that for any basic interval $I=[a,a+t]$, by the definition of
$\widehat{I}$, we have
\[
\mathcal{L}(\widehat{I})\leq(a+t)^{2}-a^{2}=t(2a+t)\leq3t=3\mathcal{L}(I).
\]
Now we have the following estimation:
\begin{align*}
& \mathcal{L}\left( K_{R}\cdot K_{R}\setminus\left( \bigcup\nolimits_{I\in
\mathcal{A}\text{,\textrm{rank}}(I)\leq k}\widehat{I}\right) \right) \\
& \leq\mathcal{L}\left( \bigcup\nolimits_{I\in\mathcal{A}\text{,\textrm{rank}%
}(I)>k}\widehat{I}\right) \\
& \leq \sum_{I\in\mathcal{A}\text{,\textrm{rank}}(I)>k}3(m\lambda)^{k+1}\\
& =3\dfrac{(m\lambda)^{k+1}}{1-m\lambda}.
\end{align*}
\end{proof}
\noindent Let
\[
\phi_{n}(\lambda)=\mathcal{L}\left(  \bigcup\nolimits_{k=0}^{n}\lambda
^{k}\left( \bigcup\nolimits_{I\in\mathcal{A}\text{,\textrm{rank}}(I)\leq
n}\widehat{I}\right) \right)  ,
\]
which is a continuous function for $\lambda\in\lbrack1/(m+1),1/m).$ Lemma
\ref{c1} implies the following result.

\begin{lemma}
\label{uni} \label{c3} If $\alpha>0$, then $\phi_{n}(\lambda)$ uniformly
converges to $\Phi(\lambda)=\mathcal{L}(K_{\lambda}\cdot K_{\lambda})$ for any
$\lambda\in\lbrack1/(m+1),1/m-\alpha).$
\end{lemma}

\begin{proof}
Note that
\[
\Phi(\lambda)=\mathcal{L}\left(  \bigcup\nolimits_{k=0}^{\infty}\lambda
^{k}(K_{R}\cdot K_{R})\right)  .
\]
It follows from Lemma \ref{c1} that
\begin{align*}
& |\Phi(\lambda)-\phi_{n}(\lambda)|\\
& \leq\left\vert \Phi(\lambda)-\mathcal{L}\left(  \bigcup\nolimits_{k=0}%
^{n}\lambda^{k}(K_{R}\cdot K_{R})\right)  \right\vert +\left\vert
\mathcal{L}\left(  \bigcup\nolimits_{k=0}^{n}\lambda^{k}(K_{R}\cdot
K_{R})\right)  -\phi_{n}(\lambda)\right\vert \\
& \leq\mathcal{L}\left(  \bigcup\nolimits_{k=n+1}^{\infty}\lambda^{k}%
(K_{R}\cdot K_{R})\right)  +\sum_{k=0}^{n}\lambda^{k}3 \dfrac{(m\lambda
)^{n+1}}{1-m\lambda}\\
& \leq\sum\limits_{k=n+1}^{\infty}\lambda^{k}+3\dfrac{1-\lambda^{n+1}%
}{1-\lambda} \dfrac{(1-m\alpha)^{n+1}}{1-m\lambda}\leq\frac{m ^{-n-1}%
}{1-\lambda}+\dfrac{3}{1-\lambda} \dfrac{(1-m\alpha)^{n+1}}{1-m\lambda}.
\end{align*}

\end{proof}

\begin{proof}
[\textbf{Proof of Theorem \ref{main2}}]$\ $

By the continuity of $\phi_{n}(\lambda)$ and Lemma \ref{uni}, it follows that
$\Phi(\lambda)$ is continuous.
\end{proof}

\section{Final remarks}
From the graph of function $\mathcal{L}(K_{\lambda}\cdot K_{\lambda})$ (Fig. 2), we pose a natural
\begin{conjecture}
The function $\mathcal{L}(K_{\lambda}\cdot K_{\lambda})$ is an increasing
function for any $\lambda\in[1/(m+1),1/m).$
\end{conjecture}

\begin{figure}[ptbh]
\centering\includegraphics[width=0.5\textwidth]{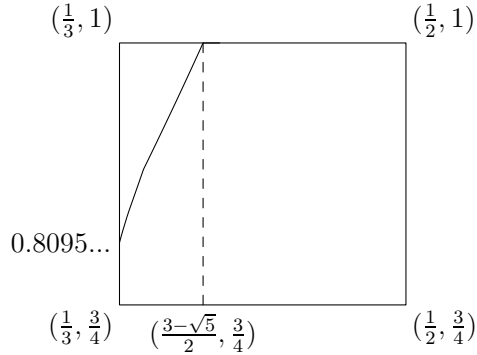}\caption{The graph of $\mathcal{L}(K_{\lambda}\cdot K_{\lambda})$ with $m=2$}
\end{figure}


\begin{thebibliography}{99}                                                                                               %
\bibitem {Tyson}Jayadev S.~Athreya, Bruce Reznick, and Jeremy T.Tyson.
\newblock Cantor set arithmetic. \newblock {\em  American
Mathematical Monthly}, 126(1):4--17, 2019.

\bibitem {Yoccoz}Carlos Gustavo~T. de~A.~Moreira and Jean-Christophe Yoccoz.
\newblock Stable intersections of regular {C}antor sets with large {H}ausdorff
dimensions. \newblock {\em Ann. of Math. (2)}, 154(1):45--96, 2001.

\bibitem {BABA1}B\'ar\'any, Bal\'azs \newblock On some non-linear projections
of self-similar sets in {$\mathbb{R}^{3}$}. \newblock {\em Fund. Math.},
237(1):83--100, 2017.

\bibitem {KarmaCor1}Jose Barrionuevo, Robert~M. Burton, Karma Dajani, and Cor
Kraaikamp. \newblock Ergodic properties of generalized {L}\"uroth series.
\newblock {\em Acta Arith.}, 74(4):311--327, 1996.

\bibitem {Bond}M.~Bond, I.~{\L }aba, and J.~Zahl. \newblock Quantitative
visibility estimates for unrectifiable sets in the plane. \newblock {\em
Trans. Amer. Math. Soc.}, 368(8):5475--5513, 2016.

\bibitem {KCor}Karma Dajani and Cor Kraaikamp. \newblock {\em Ergodic theory
of numbers}, volume~29 of \emph{Carus Mathematical Monographs}.
\newblock Mathematical Association of America, Washington, DC, 2002.

\bibitem {KM}Karma Dajani and Martijn de~Vries. \newblock Measures of maximal
entropy for random {$\beta$}-expansions. \newblock {\em J. Eur.
Math. Soc. (JEMS)}, 7(1):51--68, 2005.

\bibitem {DajaniDeVrie}Karma Dajani and Martijn de~Vries. \newblock Invariant
densities for random {$\beta$}-expansions. \newblock {\em J. Eur.
Math. Soc. (JEMS)}, 9(1):157--176, 2007.

\bibitem {KarmaOomen}Karma Dajani and Margriet Oomen. \newblock Random {$N$%
}-continued fraction expansions. \newblock {\em J. Approx. Theory}, 227:1--26, 2018.

\bibitem {FG}Kenneth Falconer. \newblock {\em Fractal geometry}.
\newblock John Wiley \& Sons, Ltd., Chichester, 1990. \newblock Mathematical
foundations and applications.

\bibitem {FW2000}De-Jun Feng, Su~Hua, and Zhi-Ying Wen. \newblock The
pointwise densities of the {C}antor measure. \newblock {\em J. Math. Anal.
Appl.}, 250(2):692--705, 2000.

\bibitem {Hochman2012}Michael Hochman and Pablo Shmerkin. \newblock Local
entropy averages and projections of fractal measures. \newblock {\em Ann. of
Math. (2)}, 175(3):1001--1059, 2012.

\bibitem {Hutchinson}John~E. Hutchinson. \newblock Fractals and
self-similarity. \newblock {\em Indiana Univ. Math. J.}, 30(5):713--747, 1981.

\bibitem {JX2018}Kan Jiang and Lifeng Xi. \newblock Interiors of continuous
images of the middle-third cantor set. \newblock {\em arXiv:1809.01880}, 2018.

\bibitem {JM}Kan Jiang, Xiaomin Ren, Jiali Zhu, and Li~Tian.
\newblock Multiple representations of real numbers on self-similar sets with
overlaps. \newblock {\em Fractals}, 27(4):1950051, 17, 2019.

\bibitem {MO}Pedro Mendes and Fernando Oliveira. \newblock On the topological
structure of the arithmetic sum of two {C}antor sets. \newblock
\emph{Nonlinearity}, 7(2):329--343, 1994.

\bibitem {HS}Hugo Steinhaus. \newblock Mowa W{\l }asno\'{s}\'{c} Mnogo\'{s}ci
Cantora. \newblock {\em Wector, 1-3. English translation in: STENIHAUS,
H.D.} 1985.

\bibitem {Kan2019}Li~Tian, Jiangwen Gu, Qianqian Ye, Lifeng Xi, and Kan Jiang.
\newblock Multiplication on self-similar sets with overlaps.
\newblock {\em J. Math. Anal. Appl.}, 478(2):357--367, 2019.

\bibitem {Palis}Jacob Palis and Floris Takens.
\newblock {\em Hyperbolicity and sensitive chaotic dynamics at homoclinic
bifurcations}, volume~35 of \emph{Cambridge Studies in Advanced Mathematics}.
\newblock Cambridge University Press, Cambridge, 1993. \newblock Fractal
dimensions and infinitely many attractors.

\bibitem {PS}Yuval Peres and Pablo Shmerkin. \newblock Resonance between
{C}antor sets. \newblock {\em Ergodic Theory Dynam. Systems}, 29(1):201--221, 2009.
\end{thebibliography}
\end{document}